\documentclass[11pt]{article}
\usepackage{amssymb}

\newtheorem{theorem}{Theorem}[section]
\newtheorem{corollary}[theorem]{Corollary}
\newtheorem{definition}[theorem]{Definition}
\newtheorem{example}[theorem]{Example}
\newtheorem{lemma}[theorem]{Lemma}
\newtheorem{proposition}[theorem]{Proposition}
\newtheorem{remark}[theorem]{Remark}
\newenvironment{proof}[1][Proof]{\textbf{#1.} }{\
\rule{0.5em}{0.5em}}
\parskip \baselineskip

\def\bea{\begin{eqnarray*}}
\def\eea{\end{eqnarray*}}

\newcommand{\Z}{\mathbb{Z}}

\newcommand{\ul}{\underline}
\newcommand{\ra}{\rightarrow}

\title{ISOMORPHISMS BETWEEN MORITA CONTEXT RINGS\footnote{The research of the
first two authors was supported by Grant ID-1904, contract
479/13.01.2009 of CNCSIS. The third author was supported by the
National Research Foundation of South Africa. Any opinions,
findings and conclusions or recommendations expressed in this
material are those of the authors and therefore the National
Research Foundation does not accept any liability in regard
thereto.}}

\author{C. Boboc$^a$, S. D\u asc\u alescu$^b$, L. van
Wyk$^c$\footnote{Corresponding author. E-mail addresses:
crinaboboc@yahoo.com (C. Boboc), sdascal@fmi.unibuc.ro and
sdascal2001@yahoo.com (S. D\u asc\u alescu), LvW@sun.ac.za (L. van
Wyk)}\\[3pt] $^a$Faculty of Physics, University of Bucharest, Bucharest,
Romania\\[3pt] $^b$Faculty of Mathematics, University of Bucharest, \\
Str.~Academiei 14, Bucharest 1, RO-70109, Romania\\[3pt] $^c$Department of
Mathematical Sciences (Mathematics Division), Stellenbosch University,\\
 Private Bag X1, Matieland 7602, Stellenbosch, South
Africa}
\date{}
\begin{document}
\maketitle

\begin{abstract}

Let $(R, S, _R\negthinspace M_S,
_S\negthinspace N_R, f, g)$ be a general Morita context, and let $T=\left[ \begin{array}{cc} R & _RM_S\\
 _SN_R & S \end{array} \right]$ be the ring associated with this context. Similarly, let
 $T'=\left[ \begin{array}{cc} R' & M'\\
 N' & S' \end{array} \right]$ be another Morita context ring. We study the set
 $\mbox{Iso}(T,T')$ of ring isomorphisms from $T$ to $T'$. Our interest in this problem is
 motivated by: (i) the problem to determine the automorphism group of
 the ring $T$, and (ii) the recovery of the non-diagonal tiles problem for this type of generalized matrix rings.
 We introduce
two classes of isomorphisms from $T$ to $T'$, the disjoint union
of which is denoted by $\mbox{Iso}_0(T,T')$. We describe
$\mbox{Iso}_0(T,T')$ by using the $\Z$-graded ring structure of
$T$ and $T'$. Our main result characterizes $\mbox{Iso}_0(T,T')$
as the set consisting of all semigraded isomorphisms and all
anti-semigraded isomorphisms from $T$ to $T'$, provided that the
rings $R'$ and~$S'$ are indecomposable and at least one of $M'$
and $N'$ is nonzero; in particular $\mbox{Iso}_0(T,T')$ contains
all graded isomorphisms and all anti-graded isomorphisms from $T$
to $T'$. We also present a situation where
$\mbox{Iso}_0(T,T')=\mbox{Iso}(T,T')$. This is in the case where
$R,S,R'$ and $S'$ are rings having only trivial idempotents and
all the Morita maps are zero. In particular, this shows that
 the group of
automorphisms of $T$ is completely determined.
\end{abstract}


\section{Introduction}

 \vskip-0.3cm\qquad Morita contexts appeared as a key ingredient in the work of Morita that
 described equivalences between full categories of modules over
 rings with identities. One of the fundamental results in this
 direction says that the categories of left modules over the rings
 $R$ and $S$ are equivalent if and only if there exists a strict
 Morita context connecting $R$ and $S$.

\vskip -0.4truecm\qquad Throughout the sequel $(R, S,
_R\negthinspace M_S, _S\negthinspace N_R, f, g)$ will be a general
Morita context, i.e.~$R$ and~$S$ are rings with identity, $M$ is a
left $R$, right $S$-bimodule, $N$ is a left $S$, right
$R$-bimodule, $f: M \otimes_S N \ra R$ is a morphism of
$R,R$-bimodules and $g: N\otimes_R M \ra S$ is a morphism of
$S,S$-bimodules, such that if we denote $[m,n] :  = f(m \otimes
n)$ and $(n,  m) : = g(n \otimes m)$, we have that
\begin{equation}
[m,n]m' = m(n, m') \quad \mbox{and} \quad n[m,n'] = (n, m)n'
\label{(1)}
\end{equation}
for all $m,m'\in M$ and all $n,n'\in N$.

\vskip -0.4truecm\qquad With such a Morita context we associate
the ring
$T = \left[ \begin{array}{cc} R & M\\
 N & S \end{array} \right]$
with operations the formal operations of $2\times 2$-matrices,
using $[,]$ and $(,)$ in defining the multiplication (see, for
example [MCR], page 12), more precisely
$$\left[ \begin{array}{cc} r & m\\
 n & s \end{array} \right]
 \left[ \begin{array}{cc} r' & m'\\
 n' & s' \end{array} \right]=\left[ \begin{array}{cc} rr'+[m,n'] & rm'+ms'\\
 nr'+sn' & (n,m')+ss' \end{array} \right]
.$$
$T$ is called the Morita context ring associated with the
given Morita context. Such rings (especially in the case where
both Morita maps $f$ and $g$ are zero, or even more particularly
when $N=0$) have been intensively used to provide examples and
counterexamples in ring theory (see [MCR]).

\vskip -0.4truecm\qquad The aim of this paper is to investigate
isomorphisms between two Morita context rings $T$ as above and $T' = \left[ \begin{array}{cc} R' & M'\\
 N' & S' \end{array} \right]$. This investigation is motivated by
 at least the following two types of problems:

 $\bullet$ {\it What are the automorphisms of a Morita context
 ring $T$?}

 $\bullet$ {\it Can we recover the tiles in Morita context rings?
 For example, if $\left[ \begin{array}{cc} R & M\\
 N & S \end{array} \right]$ and $\left[ \begin{array}{cc} R & M'\\
 N' & S \end{array} \right]$ are two isomorphic Morita context
 rings, when are $M$ and $M'$ (respectively $N$ and $N'$) isomorphic in some sense?}\\

\vskip -0.4truecm\qquad Automorphisms of various kinds of
algebraic structures have been studied extensively in the
literature. Knowing the group of automorphisms of a certain object
can be key information about it. It is a very difficult problem to
find all the automorphisms of $T$. In fact, as we explain in
Remark \ref{finalremark}, Morita context rings are up to
isomorphism just the rings with non-trivial idempotents (indeed a
very large class of rings), and so there is virtually no hope of
finding the automorphisms of $T$ in the general case. Even in a
very particular case, where the rings $R$ and $S$ are the same and
both bimodules $M$ and $N$ are also $R$, while the Morita maps are
just the multiplication of the ring, the Morita ring is the full
$2\times2$ matrix ring $\Bbb M_2(R)$, whose automorphisms are
known only for special rings $R$.

\vskip-0.4truecm\qquad The Skolem-Noether theorem (see, for
example, [R]) states that if $A$ is a simple artinian algebra
which is finite-dimensional over its centre $F$, then all
$F$-automorpisms of $A$ are inner. J\o ndrup [J1] showed that if
$A$ is a simple artinian algebra which is finite-dimensional over
its centre $F$, then all $F$-automorphisms of the ring $\Bbb
U_n(A)$ of $n\times n$ upper triangular matrices over such a ring
$A$ are also inner. (Note that when $n=1$, this is the
Skolem-Noether theorem.) J\o ndrup also computed the automorphism
groups of certain non-semiprime
rings, for example, $\Bbb U_2(k[X]),\  \left[ \begin{array}{cc} k[X] & k[X,Y]\\
 0 & k[Y] \end{array} \right]$, where $k$ is a field and $X$ and $Y$ are indeterminates.
 The automorphism groups of $k[X,Y]$ and $k\langle X,Y\rangle$ (the free algebra)
 are known, and in fact equal (see~[D]).

\vskip -0.4truecm\qquad Rosenberg and Zelinsky [RZ] obtained
information about the extent to which it can be true that not all
automorphisms of central separable algebras are inner. These
algebras constitute a class which is more general than the class
of matrix algebras $\Bbb M_n(R)$ over commutative rings $R$,
considered by Isaacs [I], who showed that, although not all
$R$-algebra automorphisms of $\Bbb M_n(R)$ are inner, the extent
of this failure is somehow under control. For example, the
commutator of any two automorphisms and the $n$th power of each of
them are inner.

\vskip -0.4truecm\qquad For any algebra $A$ over a commutative
ring $R$ such that every nontrivial $R$-algebra endomorphism of
$A$ is an $R$-algebra automorphism, Barker and Kezlan [BK] showed
that every $R$-algebra automorphism $\phi$ of $\Bbb U_n(A)$
factors as $\phi=\varphi\psi$, where $\varphi$ is inner and $\psi$
is an $R$-algebra automorphism of $\Bbb U_n(A)$ which is induced
(componentwise) from some $R$-algebra automorphism of $A$, and
they provided an example showing that $\phi$ itself need not be
inner. J\o ndrup [J2] showed that the above hypothesis on
$R$-algebra endomorphisms of $A$, used in obtaining the
factorization $\phi=\varphi\psi$, can be removed in the case where
$A$ is a prime $R$-algebra.

\vskip -0.4truecm\qquad Other papers on automorphisms of algebras
of upper triangular matrices, automorphisms of structural matrix
algebras and automorphisms of other types of subalgebras of full
matrix algebras include, for example, [C1-2], [Ke] and [Ko].

\vskip -0.4truecm\qquad As far as the possible recovery of the tiles
 is concerned, it was shown in [DW] that
it can happen that
$$\left[ \begin{array}{cc} R & R\\
R & R \end{array} \right]\simeq \left[ \begin{array}{cc} R & 0\\
R & R \end{array} \right]\simeq \left[ \begin{array}{cc} R & R\\
0 & R \end{array} \right]\simeq \left[ \begin{array}{cc} R & 0\\
0 & R \end{array} \right]$$ for certain rings $R$, so the tiles in
the non-diagonal positions of a Morita context ring cannot be
recovered in general. However, in [KDW] a positive recovery result
in this vein was obtained for a generalized
triangular matrix ring $\left[ \begin{array}{cc} R & _RM_S\\
0 & S \end{array} \right]$ over rings $R$ and $S$ having  only the
idempotents 0 and~1, in particular, over indecomposable
commutative rings or over local rings (not necessarily
commutative). In addition, the automorphism group of such a
generalized triangular matrix ring was obtained. In [AW] this
result was extended to the case where the diagonal rings $R$ and
$S$ are strongly indecomposable (not necessarily commutative)
rings, which include rings with only the trivial idempotents, as
well as endomorphism rings of vector spaces, or more generally,
semiprime indecomposable rings. We note that strongly
indecomposable rings are called semicentral reduced rings in
[BHKP].

\vskip -0.4cm\qquad In the present sequel we consider ring
isomorphisms between two Morita context rings $T$ and $T'$. We
denote the set of all such isomorphisms by $\mbox{Iso}(T,T')$. We
introduce two classes $\mbox{Iso}_0^0(T,T')$ and
$\mbox{Iso}_0^1(T,T')$ of such isomorphisms. The disjoint union of
these two classes is denoted by $\mbox{Iso}_0(T,T')$. As Remark
\ref{remaut0aut} shows, $\mbox{Iso}_0(T,T')$ is much smaller than
$\mbox{Iso}(T,T')$ if $[,]\ne0$ or $(,)\ne0$. In the case where
$T'=T$, we denote $\mbox{Iso}_0(T,T)$ by $\mbox{Aut}_0(T)$, and we
show that it is a subgroup of $\mbox{Aut}(T)$, and
$\mbox{Iso}_0^0(T,T)$ is a normal subgroup of $\mbox{Aut}_0(T)$.

\vskip -0.4cm\qquad For understanding $\mbox{Iso}_0(T,T')$ we
emphasize the structure of $T$  and $T'$ as  $\Z$-graded rings.
There are some other isomorphisms associated with this graded
structure: the set of graded isomorphisms, denoted by
$\mbox{Iso}_g^+(T,T')$, and the set of anti-graded isomorphisms,
denoted by $\mbox{Iso}_g^-(T,T')$ (which we define in Section 2).
We consider $\mbox{Iso}_g(T,T')=\mbox{Iso}_g^+(T,T')\cup
\mbox{Iso}_g^-(T,T')$. In the case where $T'=T$,
$\mbox{Aut}_g(T)=\mbox{Iso}_g(T,T)$ is a subgroup of
$\mbox{Aut}(T)$. We show that in the case where one of the Morita
contexts with which $T$ and $T'$ are associated is strict, we have
that $\mbox{Iso}_0(T,T')\subseteq \mbox{Iso}_g(T,T')$.

\vskip -0.4cm\qquad The new concepts of a semigraded isomorphism
and an anti-semigraded isomorphism are introduced in Definition
\ref{defsemigraded}, and our main result describes
$\mbox{Iso}_0(T,T')$ in terms of these isomorphisms from $T$ to
$T'$ in the case where the rings $R'$ and $S'$ in the Morita
context ring $T'$ are indecomposable and at least one of the
bimodules $M'$ and $N'$ is nonzero. To be precise,
Theorem~\ref{semigraded} says that, under these conditions,
$\mbox{Iso}_0^0(T,T')$ is the set of all semigraded isomorphisms
from $T$ to $T'$, and $\mbox{Iso}_0^1(T,T')$ is the set of all
anti-semigraded isomorphisms from $T$ to~$T'$, and so
$\mbox{Iso}_g(T,T')\subseteq \mbox{Iso}_0(T,T')$. In particular,
this result can be seen as a way to recover tiles from certain
types of isomorphisms.

\vskip -0.4cm\qquad In the last section we present a situation
where $\mbox{Iso}_0(T,T')=\mbox{Iso}(T,T')$. This is in the case
where $R,S,R'$ and $S'$ are rings having only trivial idempotents,
and all the Morita maps are zero. In particular, this shows that
the tiles can be recovered from any isomorphism (not only from
semigraded or anti-semigraded ones), and also that the group of
automorphisms of $T$ is completely determined.

\vskip -0.4cm\qquad Throughout the paper by ring we understand a
ring with identity $1\neq 0$.

\section{Two classes of isomorphisms between Morita context rings}
\label{sectiontwoclasses}

\vskip -0.3cm\qquad Let $(R, S, _R\negthinspace M_S,
_S\negthinspace N_R, f, g)$ be a  Morita context, with $[,]$ and
$(,)$ the maps defined as in the Introduction, and let $T$ be the
associated Morita context ring. Consider another Morita context
$(R', S', _{R'}\negthinspace M'_{S'}, _{S'}\negthinspace N'_{R'},
f', g')$, and for simplicity we denote also by $[,]$ and $(,)$ the
maps defined by this second Morita context, and let $T'$ be the
associated Morita context ring.

\vskip -0.3cm\qquad Recall that if $\gamma:R\ra R'$ and $\delta
\in S\ra S'$ are ring isomorphisms, we say that a morphism $u:
(M,+)\ra (M',+)$ is a $\gamma$-$\delta$-bimodule isomorphism  if
$$u (rms) = \gamma(r ) u (m) \delta (s)$$ for all $m\in M, \ r\in
R, \ s\in S.$ This is in fact equivalent to the fact that $u$ is
an isomorphism of $R$-$S$-bimodules when $M'$ is regarded as such
a bimodule via $\gamma$ and $\delta$.

\vskip -0.4cm\qquad We consider two classes of ring isomorphisms
between the Morita context rings $T$ and $T'$, constructed in the
 following two propositions.

\begin{proposition} \label{propdefaut0}
Let $(\gamma, \delta, u,v,m'_0,n'_0)$ be a 6-tuple such that
$\gamma:R\ra R'$ and $\delta:S\ra S'$ are ring isomorphisms,
$u:M\ra M'$ is a $\gamma$-$\delta$-bimodule isomorphism, \ $v:N\ra
N'$ is a $\delta$-$\gamma$-bimodule isomorphism,  and $m'_0 \in
M'$ and $n'_0 \in N'$ are fixed elements, such that
the following conditions are satisfied:\\
(i) $[m'_0, N'] =0$ and $(N', m'_0) =0$\\
(ii) $[M', n'_0] =0$ and $(n'_0, M') =0$\\
(iii) $[u(m), v(n) ] = \gamma ([m,n])$ and $(v(n), u(m)) = \delta(
(n, m))$ for all $m\in M, \ n\in N$.\\
Then the map $\phi:T\ra T'$ defined by
$$\phi \left(\left[ \begin{array}{cc} r & m\\
n &s \end{array}\right]\right) = \left[ \begin{array}{cc}\gamma
(r) & \gamma (r) m'_0 -m' _0 \delta (s) + u(m) \\ n'_0 \gamma (r)
- \delta (s) n'_0 + v(n) & \delta (s) \end{array}\right]$$ is a
ring isomorphism.
\end{proposition}

\begin{proof} The additivity (respectively the injectivity) of
$\gamma,\delta,u$ and $v$ ensure that $\phi$ is additive
(respectively injective). \vskip -0.4truecm\qquad Next, let
$r_1,r_2\in R, \ s_1,s_2\in S, \ m_1,m_2\in M, \ n_1,n_2\in N.$
For the multiplicativity of $\phi$ we only show that the entries
of
$$\phi \left(
\left[ \begin{array}{cc} r_1 & m_1\\ n_1 &s_1
\end{array}\right]\right)\phi \left(\left[
\begin{array}{cc} r_2 & m_2\\ n_2 &s_2 \end{array}\right]\right)\qquad \mbox{and} \qquad
\phi \left(\left[
\begin{array}{cc} r_1r_2 +[m_1,n_2] & r_1m_2+m_1s_2\\ n_1r_2+s_1n_2
&(n_1,m_2)+s_1s_2
\end{array}\right]\right)$$ in position $(1,2)$ are equal, since the equality in the
other positions can be checked in a similar way. These are,
respectively, \vskip -0.8truecm
$$\gamma(r_1)\bigl(\gamma(r_2)m'_0-m'_0\delta(s_2)+u(m_2)\bigr)+\bigl(\gamma(r_1)m'_0-
m'_0\delta(s_1)+u(m_1)\bigr)\delta(s_2)$$ \vskip -0.5truecm and
\vskip -0.8truecm
$$\gamma(r_1)\bigl(r_1r_2+[m_1,n_2]\bigr)m'_0-m'_0\delta\bigl((n_1,m_2)+
s_1s_2\bigr)+u(r_1m_2+m_1s_2).$$ Using the hypotheses of the
proposition, both these expressions simplify to \vskip -0.8truecm
$$\gamma(r_1)\gamma(r_2)m'_0+\gamma(r_1)u(m_2)-m'_0\delta(s_1)\delta(s_2) +
u(m_1)\delta(s_2).$$

\vskip -0.5truecm\qquad To see that $\phi$ is onto, let $\left[
\begin{array}{cc} r' & m'\\ n' &s' \end{array}\right] \in T'$. The
surjectivity of $\gamma,\delta,u$ and $v$ implies that there are
$r\in R, \ s\in S, \ m\in M$ and $n\in N$ such that $\gamma(r)=r',
\ \delta(s)=s',  \ u(m)=m'-\gamma(r)m'_0+m'_0\delta(s)$ and
$v(n)=n'-n'_0\gamma(r)+\delta(s)n'_0$, and so $\phi \left(\left[
\begin{array}{cc} r & m \\ n &s \end{array}\right]\right)$ \vskip
-0.5truecm \bea &=& \left[ \begin{array}{cc} \gamma(r) & \gamma
(r) m'_0 -m' _0 \delta (s) + (m'-\gamma(r)m'_0+m'_0\delta(s)) \\
n'_0 \gamma (r) - \delta (s) n'_0 +
(n'-n'_0\gamma(r)+\delta(s)n'_0)   & \delta(s) \end{array}\right] \\
&=& \left[ \begin{array}{cc} r' & m'\\ n' &s'
\end{array}\right],\eea which concludes the proof. \end{proof}

\begin{proposition} \label{propdefaut1} Let
$(\rho,\sigma,\mu,\nu,m'_\ast,n'_\ast)$ be a 6-tuple with $\rho: R
\ra S'$ and $\sigma : S \ra R'$ ring isomorphisms, $\mu : (M, +)
\ra (N', +)$ and $\nu : (N, +) \ra (M', +)$ group isomorphisms
such that $\mu (rms) = \rho(r ) \mu (m) \sigma (s)$ and $\nu (sn
r) = \sigma (s) \nu (n) \rho (r)$ for all $m\in M, \ n\in N, \
r\in R, \ s\in S$, and  $m'_\ast \in M'$ and $n'_\ast \in N'$ are
fixed elements, such that the following properties are
satisfied:\\
(i) $[m'_\ast, N'] =0$ and $(N', m'_\ast ) =0$\\
(ii) $[M', n'_\ast ] =0$ and $(n'_\ast, M') =0$\\
(iii) $(\mu (m), \nu (n) ) = \rho ([m,n])$ and $[\nu (n), \mu (m)
] = \sigma ((n, m))$ for all $m\in M, \ n\in N$.\\
Then the map $\psi:T\ra T'$ defined by
$$\psi \left(\left[ \begin{array}{cc} r & m \\
n & s \end{array}\right]\right) = \left[ \begin{array}{cc}\sigma
(s) & m'_\ast \rho(r) - \sigma (s) m'_\ast + \nu (n) \\
\rho(r)n'_\ast - n'_\ast \sigma (s) + \mu (m) & \rho(r)
\end{array} \right]$$ is a ring isomorphism.
\end{proposition}

\begin{proof} Similar to the proof of Proposition
1.1.\end{proof}

\qquad We denote by $\mbox{Iso}_0^0(T,T')$ and
$\mbox{Iso}_0^1(T,T')$ the sets of ring isomorphisms defined in
Proposition~\ref{propdefaut0} and Proposition \ref{propdefaut1}
respectively. Of course,  $\mbox{Iso}_0^0(T,T')$ and
$\mbox{Iso}_0^1(T,T')$ may be empty; for instance, in the case
where the ring $R$ is isomorphic neither to $R'$ nor to $S'$, both
these sets are empty. In any case we have that
$\mbox{Iso}_0^0(T,T')\cap \mbox{Iso}_0^1(T,T')=\emptyset$. We
denote
$$\mbox{Iso}_0(T,T')=\mbox{Iso}_0^0(T,T')\cup
\mbox{Iso}_0^1(T,T').$$ \emph{}In the particular case where
$T'=T$, the isomorphisms defined in Proposition~\ref{propdefaut0}
and Proposition~\ref{propdefaut1} are automorphisms of the ring
$T$. We denote $\mbox{Aut}_0^0(T)=\mbox{Iso}_0^0(T,T)$,
$\mbox{Aut}_0^1(T)=\mbox{Iso}_0^1(T,T)$ and
$\mbox{Aut}_0(T)=\mbox{Iso}_0(T,T)$. Clearly
$\mbox{Aut}_0^0(T)\neq \emptyset$, since it always contains the
identity morphism, while $\mbox{Aut}_0^1(T)$ may be empty (for
instance in the case where $R$ and $S$ are not isomorphic).
 \vskip -0.4truecm\qquad   Direct verification
yields the following two results.

\begin{proposition} \label{compositioniso}
Let $T,T',T''$ be the Morita context rings associated with the
Morita contexts $(R, S,M, N, f, g)$, $(R', S',M', N', f', g')$,
$(R'', S'',M'', N'', f'', g'')$ respectively. Let $\phi\in {\rm
Iso}_0^0(T,T')$ and $\phi' \in {\rm Iso}_0^0(T',T'')$, with $\phi$
and $\phi'$ corresponding to the 6-tuples $(\gamma, \delta, u, v,
m'_0, n'_0)$ and $(\gamma', \delta', u', v', m_0'', n_0'')$
respectively, and let $\psi \in {\rm Iso}_0^1(T,T')$ and $\psi'
\in {\rm Iso}_0^1(T',T'')$, with $\psi$ and $\psi'$ corresponding
to the 6-tuples $(\rho, \sigma, \mu, \nu, m'_\ast, n'_\ast)$ and
$(\rho', \sigma', \mu', \nu', m_\ast'', n_\ast'')$ respectively.
Then \vskip -0.4cm (i) $\phi' \circ \phi, \psi' \circ \psi \in
{\rm Iso}_0^0(T,T''), \phi^{-1}\in {\rm Iso}_0^0(T',T)$, $\phi'
\circ \psi, \psi' \circ \phi \in {\rm Iso}_0^1(T,T''), \psi^{-1}
\in {\rm Iso}_0^1(T',T)$; \vskip -0.4cm (ii) The six isomorphisms
 in (i) correspond to the following $6$-tuples respectively:
\vskip -0.4cm
$$\begin{array}{lll}
\phi' \circ \phi & \leftrightarrow &
(\gamma' \circ \gamma, \delta' \circ \delta, u'\circ u, v' \circ v, m''_0+u' (m'_0), n''_0 + v'(n'_0))\\
\\
\psi' \circ \psi & \leftrightarrow & (\sigma' \circ \rho, \rho'
\circ \sigma, \nu' \circ \mu, \mu' \circ \nu,
\nu' (n'_\ast) - m_\ast'', \mu' (m'_\ast ) -n_\ast'' ) \\
\\
\phi^{-1} & \leftrightarrow & (\gamma^{-1}, \delta^{-1}, u^{-1}, v^{-1},u^{-1} (-m'_0), v^{-1}(-n'_0))\\
\\
\phi' \circ \psi & \leftrightarrow &
(\delta' \circ \rho, \gamma' \circ \sigma, v' \circ \mu, u' \circ \nu, u' (m_\ast' ) -m''_0, v'(n_\ast') - n''_0)\\
\\
\psi' \circ \phi & \leftrightarrow &
(\rho' \circ \gamma, \sigma' \circ \delta, \mu' \circ u, \nu' \circ v, \nu' (n'_0) + m_\ast'', \mu' (m'_0) + n_\ast''))\\
\\
\psi^{-1} & \leftrightarrow & (\sigma^{-1}, \rho^{-1}, \nu^{-1},
\mu^{-1}, \mu^{-1} (n_\ast' ), \nu^{-1} (m_\ast' )).
\end{array}$$
\end{proposition}

\begin{corollary}
${\rm Aut}_0(T)$ is a subgroup of ${\rm Aut}(T)$ and ${\rm
Aut}_0^0(T)$ is a normal subgroup of ${\rm Aut}_0(T)$.
\end{corollary}

\qquad Let us note that if $(R, S, M, N, f, g)$ is a Morita context, and $T=\left[ \begin{array}{cc} R &M\\
 N & S \end{array} \right]$ is the associated Morita context ring,
 then $(S,R, N,M,g,f)$ is also a Morita context, and we denote by $T^{-1}=\left[ \begin{array}{cc} S &N\\
 M & R \end{array} \right]$ its associated Morita context ring.
 Clearly the map $$\alpha_T:T\ra T^{-1},  \ \ \alpha_T\left(\left[ \begin{array}{cc} r & m\\
n &s \end{array}\right]\right)=\left[ \begin{array}{cc} s & n\\
m &r \end{array}\right]$$ is a ring isomorphism. In fact,
$\alpha_T\in \mbox{Iso}_0^1(T,T^{-1})$. Moreover, we have that
$\alpha_T^{-1}=\alpha_{T^{-1}}$. Hence, it is easy to see from
Proposition \ref{compositioniso}  that

\begin{equation} \label{linkiso0iso1}
\mbox{Iso}_0^1(T,T')=\mbox{Iso}_0^0(T^{-1},T')\alpha_T=\alpha_{T'}^{-1}\mbox{Iso}_0^0(T,(T')^{-1})
\end{equation}

\qquad The following shows that in general $\mbox{Iso}_0(T,T')\neq
\mbox{Iso}(T,T')$.

\begin{remark} \label{remaut0aut}
Let $(R, S,M, N, f, g)$ and $(R', S',M', N', f', g')$ be Morita
contexts and consider the associated Morita context rings $T$ and
$T'$. Assume that $T\simeq T'$. We show that if  $\;{\rm
Iso}_0(T,T')={\rm Iso}(T,T')$, the set of all ring isomorphisms
from $T$ to $T'$, then necessarily $f,g,f'$ and $g'$ are~0 (thus
all the maps $[,]$ and $(,)$ from the two contexts are zero).
Indeed, we also have that ${\rm Iso}_0(T',T)={\rm Iso}(T',T)$.
Since $T\simeq T'$, we have that \bea {\rm Aut}(T)&=&\{
\phi_1\circ \phi_2\;|\; \phi_1\in
{\rm Iso}(T',T), \phi_2\in {\rm Iso}(T,T')\}\\
&=&\{ \phi_1\circ \phi_2\;|\; \phi_1\in {\rm Iso}_0(T',T),
\phi_2\in {\rm Iso}_0(T,T')\}\\
&\subseteq &{\rm Aut}_0(T),\eea and so ${\rm Aut}_0(T)={\rm
Aut}(T)$.

\vskip -0.5truecm\qquad Now let $m\in
M$. Then $\left[ \begin{array}{cc} 1&m  \\
0 & 1
\end{array}\right]\in U(T)$, and let $\eta\in {\rm Aut}(T)$ be
the associated inner automorphism. Then

\bea \eta \left(\left[ \begin{array}{cc} 0&0  \\
n & 0
\end{array}\right]\right)&=&\left[ \begin{array}{cc} 1&m  \\
0 & 1
\end{array}\right]\left[ \begin{array}{cc} 0&0  \\
n & 0
\end{array}\right]\left[ \begin{array}{cc} 1&-m  \\
0 & 1
\end{array}\right]\\
&=&\left[ \begin{array}{cc} [m,n]&-[m,n]m  \\
n & -(n,m)
\end{array}\right]\eea for every $n\in N$. As ${\rm Aut}_0(T)={\rm Aut}(T)$,
$\eta$ must be either in ${\rm Aut}_0^0(T)$ or in ${\rm
Aut}_0^1(T)$, and so
$\eta\left(\left[ \begin{array}{cc} 0&0  \\
N & 0
\end{array}\right]\right)$ is either contained in  $\left[ \begin{array}{cc} 0&0  \\
N & 0
\end{array}\right]$ or in $\left[ \begin{array}{cc} 0&M  \\
0 & 0
\end{array}\right]$. This shows that $[m,n]=0$ and $(n,m)=0$ for all $m\in M$ and all $n\in
N$. Thus $f=0$ and $g=0$. By symmetry one gets also that $f'=0$
and $g'=0$.

\end{remark}

\section{Isomorphisms associated with the graded structure of
Morita context rings}\label{sectiongraded}

\vskip -0.3cm\qquad We keep the notation as in Section
\ref{sectiontwoclasses}. For basic concepts about graded rings we
refer to [NVO]. A ring $A$ is $\Z$-graded if
$A=\oplus_{i\in{\Z}}A_i$, a direct sum of additive subgroups, such
that $A_iA_j\subseteq A_{i+j}$ for any $i,j\in \Z$.

\vskip -0.3cm\qquad The Morita context ring $T=\left[ \begin{array}{cc} R&M  \\
N & S
\end{array}\right]$ has a structure of
a $\Z$-graded ring with homogeneous components

$$T_{-1}=\left[ \begin{array}{cc} 0&0  \\
N & 0
\end{array}\right],\;\; T_0=\left[ \begin{array}{cc} R&0  \\
0 & S
\end{array}\right],\;\; T_1=\left[ \begin{array}{cc} 0&M  \\
0 & 0
\end{array}\right]$$
and $T_i=0$ for every $i\notin \{-1,0,1\}$. Therefore if $T'$ is
another Morita context ring, we can consider two classes of ring
isomorphisms from $T$ to $T'$, related to the graded structure, as
follows:

$\bullet$ {\it Graded isomorphisms}, which are isomorphisms
$\phi:T\ra T'$ such that $\phi(T_i)\subseteq T'_i$ for all $i\in
{\Z}$. We denote by $\mbox{Iso}_g^+(T,T')$ the set of all such
isomorphisms. \vskip -0.5truecm $\bullet$ {\it Anti-graded
isomorphisms}, which are isomorphisms $\phi:T\ra T'$ such that
$\phi(T_i)\subseteq T'_{-i}$ for all $i\in {\Z}$. We denote by
$\mbox{Iso}_g^-(T,T')$ the set of all such isomorphisms.

\qquad We denote by $\mbox{Iso}_g(T,T')=\mbox{Iso}_g^+(T,T')\cup
\mbox{Iso}_g^-(T,T')$. It is clear that if $M\neq 0$ or $N\neq 0$,
then the sets $\mbox{Iso}_g^+(T,T')$ and $\mbox{Iso}_g^-(T,T')$
are disjoint. If $M=0$ and $N=0$, we have that
$\mbox{Iso}_g^+(T,T')=\mbox{Iso}_g^-(T,T')$. Obviously,
$\mbox{Iso}_g(T,T')$ can be empty, for instance in the case where
the rings $T$ and $T'$ are not isomorphic.

\vskip -0.5cm\qquad In the particular case where $T'=T$, we use
the notation $\mbox{Aut}_g^+(T)=\mbox{Iso}_g^+(T,T)$,
$\mbox{Aut}_g^-(T)=\mbox{Iso}_g^-(T,T)$ and
$\mbox{Aut}_g(T)=\mbox{Aut}_g^+(T)\cup \mbox{Aut}_g^-(T)$. We note
that $\mbox{Aut}_g^+(T)$ is always non-empty, since it contains
the identity morphism, while $\mbox{Aut}_g^-(T)$ may be empty. It
is easy to see that $\mbox{Aut}_g(T)$ is a subgroup of
$\mbox{Aut}(T)$, and $\mbox{Aut}_g^+(T)$ is a normal subgroup of
$\mbox{Aut}_g(T)$.

\vskip -0.5truecm\qquad We recall that a Morita context
$(R,S,M,N,f,g)$ is strict if $f$ and $g$ are surjective, and this
implies that $f$ and $g$ are isomorphisms. In this case the
bilinear maps $[,]$ and $(,)$ are left and right non-degenerate
(i.e. for example if $[m,N]=0$ for some $m\in M$, then $m=0$). Let
us note that if $\mbox{Iso}_0(T,T')\neq \emptyset$, then the
Morita context with which $T$ is associated is strict if and only
if so is the Morita context with which $T'$ is associated. Indeed,
if there exists $\phi \in \mbox{Iso}_0^0(T,T')$, associated with
the 6-tuple $(\gamma, \delta, u,v,m'_0,n'_0)$, then the relations
$[u(m), v(n) ] = \gamma ([m,n])$ and $(v(n), u(m)) = \delta( (n,
m))$, combined with the surjectivity of $u,v,\gamma,\delta$, show
that $f$ and $g$ are surjective if and only if $f'$ and $g'$ are
surjective. Similarly in the case where there exists $\phi \in
\mbox{Iso}_0^1(T,T')$.

\begin{proposition}  \label{autgstrict}
If the Morita context with which $T$ (or $T'$) is associated is
strict, then ${\rm Iso}_0(T,T')\subseteq {\rm Iso}_g(T)$.
\end{proposition}
\begin{proof} If $\mbox{Iso}_0(T,T')=\emptyset$, the result is clear. If
$\mbox{Iso}_0(T,T')\neq \emptyset$, then by the remark preceding
this proposition, the Morita context with which $T$ is associated
is strict if and only if so is the Morita context with which $T'$
is associated. Therefore we can assume that both these contexts
are strict.

\vskip -0.5truecm\qquad Let $\phi \in \mbox{Iso}_0^0(T,T')$ be
associated with the 6-tuple $(\gamma, \delta, u,v,m'_0,n'_0)$.
Since the Morita context is strict, the condition $[m'_0,N']=0$
implies that $m'_0=0$. Similarly, $n'_0=0$. Then
$\phi \left(\left[ \begin{array}{cc} r & m\\
n &s \end{array}\right]\right) = \left[ \begin{array}{cc}\gamma (r) & u(m) \\
v(n) & \delta (s)
\end{array}\right]$, and so $\phi
\in \mbox{Iso}_g^+(T,T')$. Similarly, any $\phi \in
\mbox{Iso}_0^1(T,T')$ is
of the form $\phi \left(\left[ \begin{array}{cc} r & m \\
n & s \end{array}\right]\right) = \left[ \begin{array}{cc}\sigma (s) & \nu (n) \\
\mu (m) & \rho(r) \end{array} \right]$ for some
$\rho,\sigma,\mu,\nu$, and so $\phi \in
\mbox{Iso}_g^-(T,T')$.\end{proof}

\qquad We recall that a ring $A$ is called indecomposable if it is
not isomorphic to a direct product of two rings with identity;
this is equivalent to the fact that the only central idempotents
of $A$ are 0 and 1. We will need the following simple fact.

\begin{lemma} \label{centralidempotents}
Let $T = \left[ \begin{array}{cc} R & M\\
 N & S \end{array} \right]$ be a Morita context ring. Then the
 central idempotent elements of $T$ are the matrices of the form $\left[ \begin{array}{cc} r & 0 \\
0  & s \end{array} \right]$, where $r$ and $s$ are central
idempotents of $R$ and~$S$ respectively, such that $rm=ms$ for all
$m\in M$, and $sn=nr$ for all $n\in N$.
\end{lemma}
\begin{proof} By looking at the
commutation relations with $\left[ \begin{array}{cc} 1 & 0 \\
0  & 0 \end{array} \right]$, we see that a central element of~$T$
must be of the form $X=\left[ \begin{array}{cc} r & 0 \\
0  & s \end{array} \right]$, and it is clear that $r$ must be
central in $R$, and $s$ must be central in $S$. If moreover $X$ is
an idempotent, then $r$ and $s$ are central idempotents in $R$
and~$S$, respectively. Looking at the commutation relations with
an
elements of the form~$\left[ \begin{array}{cc} 0 & m \\
n  & 0 \end{array} \right]$, we get that $rm=ms$ for all $m\in M$,
and $sn=nr$ for all $n\in N$. Clearly any such matrix
$\left[ \begin{array}{cc} r & 0 \\
0  & s \end{array} \right]$ is a central idempotent of $T$.
\end{proof}

\qquad An immediate consequence is the following.

\begin{corollary} \label{corindecomposable}
Let $T = \left[ \begin{array}{cc} R & M\\
 N & S \end{array} \right]$ be a Morita context ring such that $R$
 and $S$ are indecomposable rings and not both $M$ and $N$ are
 zero. Then $T$ is an indecomposable ring.
 \end{corollary}

\vskip -0.5truecm\qquad Another consequence of Lemma
\ref{centralidempotents} shows that
 $T$ can be indecomposable without both $R$ and $S$ being so.

 \begin{example}
 Let $R$ be a ring which is not indecomposable, and $S=\Z$. Let
 $M=R$, regarded as an $R$-$\Z$-bimodule, and $N=0$. Then using Lemma \ref{centralidempotents},
 it is easy to see that $T$ does not have non-trivial central
 idempotents, and so it is indecomposable.
 \end{example}

\qquad At this point we introduce two new types of isomorphisms
related to the graded structure of $T$ and $T'$.

\begin{definition} \label{defsemigraded}
Let $\phi:T\ra T'$ be an isomorphism between the Morita context
rings $T$ and $T'$. Then $\phi$ is
called\\
 $\bullet$ a semigraded isomorphism if $\phi(T_i)\subseteq T'_i$ for
all $i\in \{-1,1\}$. \\
$\bullet$ an anti-semigraded isomorphism if $\phi(T_i)\subseteq
T_{-i}'$ for all $i\in \{-1,1\}$.
\end{definition}

\vskip -0.5truecm \qquad Now we can prove the main result of this
section.

\begin{theorem}  \label{semigraded}
Let $T = \left[ \begin{array}{cc} R & M\\
 N & S \end{array} \right]$ and $T' = \left[ \begin{array}{cc} R' & M'\\
 N' & S' \end{array} \right]$ be Morita context rings such that $R'$ and $S'$ are
 indecomposable rings, and at
least one of $M'$ and $N'$ is nonzero. Then the following
assertions hold: \\
(1) ${\rm Iso}_0^0(T,T')$ is the set of all semigraded
isomorphisms from $T$ to $T'$.\\
(2) ${\rm Iso}_0^1(T,T')$ is the set of all anti-semigraded
isomorphisms from $T$ to $T'$. \\
In particular,
 ${\rm Iso}_g(T,T')\subseteq {\rm Iso}_0(T,T')$.
\end{theorem}
\begin{proof}
(1) It is clear that any $\phi \in \mbox{Iso}_0^0(T,T')$ is a
semigraded isomorphism.

\vskip -0.3truecm\qquad Let $\phi$ be a semigraded isomorphism
from $T$ to $T'$. Then there exist $u:M\ra M'$ and $v:N\ra N'$
such that

\vskip -0.5truecm

\begin{equation}\label{phim}
\phi \left(\left[ \begin{array}{ll} 0 & m\\
0 &0 \end{array}\right]\right) = \left[ \begin{array}{cc}0 & u(m) \\
0 & 0
\end{array}\right]\end{equation}
for every $m\in M$, and

\begin{equation} \label{phin}
\phi \left(\left[ \begin{array}{ll} 0 & 0\\
n &0 \end{array}\right]\right) = \left[ \begin{array}{cc}0 & 0 \\
v(n) & 0
\end{array}\right]
\end{equation}
for every $n\in N$. Clearly $u$ and $v$ are injective additive
morphisms.

\vskip -0.2truecm\qquad Let $\phi \left(\left[ \begin{array}{ll} 1 & 0\\
0 &0 \end{array}\right]\right) = \left[ \begin{array}{cc}r'_0 & m'_0 \\
n'_0 & s'_0
\end{array}\right]$. Since $ \left[ \begin{array}{ll} 1 & 0\\
0 &0 \end{array}\right]$ is an idempotent in $T$, we have
that $\left[ \begin{array}{cc}r'_0 & m'_0 \\
n'_0 & s'_0
\end{array}\right]$ is an idempotent in $T'$, which in particular
implies that

\vskip -0.8truecm

\begin{equation}\label{r0squared}
(r'_0)^2+[m'_0,n'_0]=r'_0
\end{equation}

\vskip -0.5truecm and \vskip -0.8truecm

\begin{equation}\label{s0squared}
(s'_0)^2+(n'_0,m'_0)=s'_0.
\end{equation}

\vskip-0.3truecm\qquad Since $\left[ \begin{array}{ll} 1 & 0\\
0 &0 \end{array}\right]T\left[ \begin{array}{ll} 0 & 0\\
0 &1 \end{array}\right]=\left[ \begin{array}{ll} 0 & M\\
0 &0 \end{array}\right]$, by applying $\phi$ we get that
\begin{equation} \label{new1}
\left[ \begin{array}{cc}r'_0 & m'_0 \\
n'_0 & s'_0
\end{array}\right]T'\left[ \begin{array}{cc}1-r'_0 & -m'_0 \\
-n'_0 & 1-s'_0
\end{array}\right]\subseteq \left[ \begin{array}{cc} 0 & M'\\
0 &0 \end{array}\right].
\end{equation}

\vskip -0.2truecm\qquad Equation (\ref{new1}) also shows that for
every $m'\in M'$,
$$\left[ \begin{array}{cc}r'_0 & m'_0 \\
n'_0 & s'_0
\end{array}\right]\left[ \begin{array}{cc} 0 & m'\\
0 &0 \end{array}\right]\left[ \begin{array}{cc}1-r'_0 & -m'_0 \\
-n'_0 & 1-s'_0
\end{array}\right]=\left[ \begin{array}{cc} -r'_0[m',n'_0] & r'_0m'(1-s'_0)\\
-(n'_0,m')n'_0 &(n'_0,m')(1-s'_0) \end{array}\right]\in \left[ \begin{array}{cc} 0 & M'\\
0 &0 \end{array}\right],$$ implying that

\vskip -0.8truecm

\begin{equation} \label{new2}
r'_0[m',n'_0]=0
\end{equation}

\vskip -0.5truecm and \vskip -0.8truecm

\begin{equation}\label{new3}
(n'_0,m')(1-s'_0)=0.
\end{equation}
\vskip -0.5truecm

Also, we have that for every $n'\in N'$, \vskip -0.8truecm
$$\left[ \begin{array}{cc}r'_0 & m'_0 \\
n'_0 & s'_0
\end{array}\right]\left[ \begin{array}{cc} 0 & 0\\
n' &0 \end{array}\right]\left[ \begin{array}{cc}1-r'_0 & -m'_0 \\
-n'_0 & 1-s'_0
\end{array}\right]=\left[ \begin{array}{cc} [m'_0,n'](1-r'_0) & -[m'_0,n']m'_0\\
s'_0n'(1-r'_0) &-s'_0(n',m'_0) \end{array}\right]\in \left[ \begin{array}{cc} 0 & M'\\
0 &0 \end{array}\right],$$ implying that

\vskip -0.8truecm

\begin{equation} \label{new2'}
[m'_0,n'](1-r'_0)=0
\end{equation}

\vskip -0.5truecm and \vskip -0.8truecm

\begin{equation}\label{new3'}
s'_0(n',m'_0)=0.
\end{equation}
\vskip -0.5truecm

Similarly, by using the fact that

\vskip -0.8truecm
$$\left[ \begin{array}{cc}1-r'_0 & -m'_0 \\
-n'_0 & 1-s'_0
\end{array}\right]T'\left[ \begin{array}{cc}r'_0 & m'_0 \\
n'_0 & s'_0
\end{array}\right]\subseteq \left[ \begin{array}{cc} 0 & 0\\
N' &0 \end{array}\right],$$ we obtain that

\vskip -0.8truecm

\begin{equation} \label{new4}
[m'_0,n']r'_0=0
\end{equation}

\vskip -0.5truecm and \vskip -0.8truecm

\begin{equation}\label{new5}
(1-s'_0)(n',m'_0)=0,
\end{equation}

\vskip -0.5truecm as well as that \vskip -0.8truecm

\begin{equation} \label{new4'}
(1-r'_0)[m',n'_0]=0
\end{equation}

\vskip -0.5truecm and \vskip -0.8truecm

\begin{equation}\label{new5'}
(n'_0,m')s'_0=0
\end{equation}

\vskip -0.5truecm for every $m'\in M'$ and $n'\in N'$.

\vskip -0.3truecm\qquad Adding in pairs relations (\ref{new2}) and
(\ref{new4'}), (\ref{new3}) and (\ref{new5'}), (\ref{new2'}) and
(\ref{new4}), (\ref{new3'}) and (\ref{new5}), we see that

\vskip -0.8truecm

\begin{equation}\label{new14}
[m',n'_0]=0, \quad (n'_0,m')=0, \quad [m'_0,n']=0, \quad
(n',m'_0)=0,
\end{equation}
for every $m'\in M'$ and $n'\in N'$.

\vskip -0.3truecm\qquad Now if we use (\ref{new14}) in
(\ref{r0squared}) and (\ref{s0squared}), we find that $r'_0$ is an
idempotent of $R'$, and $s'_0$ is an idempotent of $S'$.

\vskip -0.3truecm\qquad If we apply $\phi$ to the relation $\left[ \begin{array}{cc} 1 & 0\\
0 &0 \end{array}\right] \left[ \begin{array}{cc} 0 & m\\
0 &0 \end{array}\right]=\left[ \begin{array}{cc} 0 & m\\
0 &0 \end{array}\right]$, where $m\in M$, we obtain that


\begin{equation}\label{new7}
r'_0u(m)=u(m)
\end{equation}
for every $m\in M$.

\qquad Similarly, if we apply $\phi$ to the relation $\left[ \begin{array}{cc} 0 & m\\
0 &0 \end{array}\right] \left[ \begin{array}{cc} 1 & 0\\
0 &0 \end{array}\right]=0$, we obtain that


\begin{equation} \label{new9}
u(m)s'_0=0
\end{equation}
for every $m\in M$.

\vskip -0.3truecm \qquad Similarly, we get that

\vskip -0.8truecm

\begin{equation} \label{new12}
v(n)r'_0=v(n)
\end{equation}

\vskip -0.5truecm and \vskip -0.8truecm

\begin{equation} \label{new13}
s'_0v(n)=0
\end{equation}
for every $n\in N$.

\vskip -0.5truecm \qquad Now fix some $a'\in R'$, $b'\in S'$.
Since $\phi$ is surjective, there
exists $\left[ \begin{array}{cc} r & p\\
q &s \end{array}\right]\in T$ such that $\phi \left(\left[ \begin{array}{cc} r & p\\
q &s \end{array}\right]\right)=\left[ \begin{array}{cc} a' & 0\\
0 &b' \end{array}\right]$. Then $\phi \left(\left[ \begin{array}{cc} r & 0\\
0 &s \end{array}\right]\right)=\left[ \begin{array}{cc} a' & -u(p)\\
-v(q) &b' \end{array}\right]$. Apply $\phi$ to the relation
$$\left[ \begin{array}{cc} 1 & 0\\
0 &0 \end{array}\right]\left[ \begin{array}{cc} r & 0\\
0 &s \end{array}\right]=\left[ \begin{array}{cc} r & 0\\
0 &s \end{array}\right]\left[ \begin{array}{cc} 1 & 0\\
0 &0 \end{array}\right]$$ and get
$$\left[ \begin{array}{cc}r'_0 & m'_0 \\
n'_0 & s'_0
\end{array}\right]\left[ \begin{array}{cc} a' & -u(p)\\
-v(q) &b' \end{array}\right]=\left[ \begin{array}{cc} a' & -u(p)\\
-v(q) &b' \end{array}\right]
\left[ \begin{array}{cc}r'_0 & m'_0 \\
n'_0 & s'_0
\end{array}\right].$$
In particular, this implies that
$$r'_0a'-[m'_0,v(q)]=a'r'_0-[u(p),n'_0],$$
 which in view of (\ref{new14}) shows that $r'_0a'=a'r'_0$. Also
we get that $$ -(n'_0,u(p))+s'_0b'=-(v(q),m'_0)+b's'_0,$$ which
shows that $s'_0b'=b's'_0$. Thus $r'_0$ is a central idempotent in
$R'$ and $s'_0$ is a central idempotent in $S'$. Since $R'$ and
$S'$ are indecomposable rings, we conclude that $r'_0$ must be
either 0 or 1 and $s'_0$ must be either 0 or 1.

\vskip -0.4truecm\qquad On the other hand, $T'$ is an
indecomposable ring by Corollary \ref{corindecomposable}. This
implies that not both $M$ and $N$ are zero, otherwise $T\simeq
R\times S$, which is not indecomposable.

\vskip -0.4truecm\qquad Since $r'_0u(m)=u(m)$ for all $m\in M$ by
(\ref{new7}), $v(n)r'_0=v(n)$ for all $n\in N$ by (\ref{new12}),
and at least one of $M$ and $N$ is nonzero, we can not have
$r'_0=0$ (otherwise $u$ or $v$ could not be injective); therefore,
$r'_0=1$. Similarly, since $u(m)s'_0=0$ for all $m\in M$ by
(\ref{new9}), and $s'_0v(n)=0$ for all $n\in N$ by (\ref{new13}),
we must have $s'_0=0$. We have obtained that
$$\phi \left(\left[ \begin{array}{cc} 1 & 0\\
0 &0 \end{array}\right]\right) = \left[ \begin{array}{cc}1 & m'_0 \\
n'_0 & 0
\end{array}\right].$$

\vskip -0.3truecm \qquad Next, let $r \in R$. Then
$$\phi \left(\left[ \begin{array}{cc} r & 0 \\
0  & 0 \end{array} \right]\right) = \left[ \begin{array}{cc} r'_1 & m'_1 \\
n'_1 & s'_1 \end{array}\right]$$ for some $r'_1 \in R', m'_1 \in
M', n'_1 \in N'$ and $s'_1 \in S'$. Apply $\phi$ to the relations
$$\left[ \begin{array}{cc} r & 0 \\
0 &0 \end{array}\right] \ \left[ \begin{array}{cc} 1 & 0 \\
0 &0 \end{array}\right] = \left[ \begin{array}{cc} r & 0 \\
0 &0 \end{array}\right] = \left[ \begin{array}{cc} 1 &0 \\
0 & 0 \end{array}\right] \ \left[ \begin{array}{cc} r & 0 \\
0 &0 \end{array}\right],$$ and obtain that
$$\left[ \begin{array}{cc} r'_1  & m'_1 \\
n'_1 & s'_1 \end{array}\right] \ \left[ \begin{array}{cc} 1 & m'_0\\
n'_0 &0 \end{array}\right] = \left[ \begin{array}{cc} r'_1 & m'_1 \\
n'_1& s'_1 \end{array}\right] = \left[ \begin{array}{cc} 1 & m'_0 \\
n'_0 & 0 \end{array}\right] \ \left[ \begin{array}{cc} r'_1 & m'_1 \\
n'_1 & s'_1 \end{array}\right],$$ and so, using (\ref{new14}), we
obtain that
$$
\left[ \begin{array}{cc} r'_1  & r'_1 m'_0 \\
n'_1 + s'_1 n'_0 & 0 \end{array}\right] = \left[ \begin{array}{cc} r'_1 & m'_1 \\
n'_1 & s'_1 \end{array}\right] = \left[ \begin{array}{cc} r'_1  & m'_1 + m'_0s'_1 \\
n'_0r'_1 & 0\end{array} \right].$$ This shows that $$ m'_1 = r'_1
m'_0,\; \; s'_1 = 0,\; n'_1 = n'_0 r'_1,$$ and so
$$\phi \left(\left[ \begin{array}{cc} r& 0 \\
0 & 0 \end{array}\right]\right) = \left[ \begin{array}{cc} r'_1 & r'_1m'_0 \\
n'_0r'_1 & 0 \end{array}\right].$$ Consequently, if we define
$\gamma : R \ra R'$ by $\gamma (r) = r'_1$, then
\begin{equation}
\phi \left(\left[ \begin{array}{cc} r &0 \\
0 &0 \end{array}\right]\right) = \left[ \begin{array}{cc} \gamma (r) & \gamma (r) m'_0\\
n'_0 \gamma (r) & 0 \end{array} \right] \label{24}
\end{equation}
for all $r \in R$. Moreover, it is clear from (\ref{24}) that
$\gamma$ is additive and injective.

\vskip -0.4truecm\qquad Applying $\phi$ to
$$\left[ \begin{array}{ll} r_1 & 0 \\
0 &0 \end{array}\right] \ \left[ \begin{array}{cc} r_2 &0 \\
0  &0 \end{array}\right] = \left[ \begin{array}{cc} r_1r_2 & 0 \\
0 & 0 \end{array}\right], \quad r_1, r_2 \in R,$$ we get
$$\left[ \begin{array}{cc} \gamma (r_1) & \gamma (r_1) m'_0 \\
n'_0 \gamma (r_1) & 0 \end{array}\right] \ \left[ \begin{array}{cc} \gamma (r_2) & \gamma (r_2)m'_0 \\
n'_0 \gamma (r_2) & 0 \end{array}\right] = \left[ \begin{array}{cc} \gamma (r_1r_2) & \gamma (r_1r_2) m'_0 \\
n'_0 \gamma (r_1r_2) & 0 \end{array}\right].$$ Looking at position
$(1,1)$, and taking into account that  $[\gamma (r_1) m'_0, n'_0
\gamma (r_2) ] = \gamma (r_1)[m'_0,n'_0]\gamma (r_2)=0$, we find
that $\gamma$ is multiplicative. Therefore $\gamma$ is an
injective ring morphism.

\vskip -0.5truecm \qquad Now
$$\phi \left(\left[ \begin{array}{cc} 0 &0 \\
0 &1 \end{array}\right]\right)= \left[ \begin{array}{cc} 0 & -m'_0 \\
-n'_0 & 1 \end{array} \right],$$ and so similar arguments as above
provide a map $\delta : \ S \ra S'$ such that
\begin{equation}
\phi \left(\left[ \begin{array}{cc} 0 & 0 \\ 0 & s \end{array}
\right]\right) =
\left[ \begin{array}{cc} 0 & -m'_0 \delta (s) \\
-\delta (s) n'_0 & \delta (s) \end{array} \right] \label{25}
\end{equation}
for all $s \in S$. Also, we obtain in a similar way to $\gamma$
that $\delta$ is an injective ring morphism.

\vskip -0.3truecm\qquad We conclude from (\ref{phim}),
(\ref{phin}), (\ref{24}) and (\ref{25}) that
\begin{equation}
\phi \left(\left[ \begin{array}{cc} r & m \\
n & s \end{array} \right] \right)= \left[ \begin{array}{cc} \gamma
(r)
& \gamma (r) m'_0 - m'_0 \delta (s) + u (m) \\
n'_0 \gamma (r) - \delta (s) n'_0 + v(n) & \delta (s)
\end{array} \right] \label{26}
\end{equation}
for all $r \in R, s \in S, m \in M, n \in N$.

\vskip -0.2truecm \qquad We have that $\gamma$ and $\delta$ are
surjective. Indeed, let $r' \in R', \ s'\in S'$.
Then there exists $\left[ \begin{array}{cc} r & m \\
n &  s \end{array} \right] \in T$ such that
$$
\left[ \begin{array}{cc} r' & 0\\
0 & s' \end{array} \right] = \phi \left(\left[ \begin{array}{cc} r & m \\
n & s \end{array} \right] \right)= \left[ \begin{array}{cc} \gamma (r)  & \gamma (r) m'_0 - m'_0 \delta (s) + u(m) \\
n'_0 \gamma (r) - \delta(s) n'_0 + v(n) & \delta (s)
\end{array}\right], $$
and so $\gamma(r)=r'$ and $\delta(s)=s'$. Therefore $\gamma:R\ra
R'$ and $\delta:S\ra S'$ are ring isomorphisms.

\vskip -0.5truecm \qquad We also have that $u$ and $v$ are
surjective. Indeed, for $m' \in
M'$ and $n'\in N'$ there is $\left[ \begin{array}{cc} r & m \\
n &  s \end{array} \right] \in T$ such that
$$
\left[ \begin{array}{cc} 0 & m'\\
n' & 0 \end{array} \right] = \phi \left(\left[ \begin{array}{cc} r & m \\
n & s \end{array} \right] \right)= \left[ \begin{array}{cc} \gamma (r)  & \gamma (r) m'_0 - m'_0 \delta (s) + u(m) \\
n'_0 \gamma (r) - \delta(s) n'_0 + v(n) & \delta (s)
\end{array}\right]. $$
Since $\gamma$ and $\delta$ are injective, we get that $r=0$ and
$s=0$, and then $m'=u(m)$ and $n'=v(n)$.

\vskip -0.4truecm \qquad Next we show that $u$ is a
$\gamma$-$\delta$-bimodule isomorphism, i.e.~$u(rms) = \gamma (r)
u(m) \delta (s)$ for all $r \in R, \ m \in M, \ s \in S$, and $v$
is a $\delta$-$\gamma$-bimodule isomorphism. Let $r \in R$, and
let $m \in M$. If we apply $\phi$ to
$$\left[ \begin{array}{cc} r&0 \\
0 & 0 \end{array} \right] \ \left[ \begin{array}{cc} 0 & m\\
0  & 0 \end{array}\right] = \left[ \begin{array}{cc} 0 & rm\\
0 &0  \end{array}\right],$$ then by (23)
$$\left[ \begin{array}{cc}
\gamma (r) & \gamma (r) m'_0 \\
n'_0 \gamma (r)& 0 \end{array}\right] \ \left[ \begin{array}{cc} 0  & u(m) \\
0 & 0 \end{array}\right] = \left[ \begin{array}{cc} 0 & u(rm)\\0 &
0 \end{array} \right],$$ and so it follows from position $(1,2)$
that
$$
u(rm) = \gamma (r) u(m).$$ Similarly,
$$
u(ms) =u(m) \delta (s), \quad v (sn) = \delta (s) v(n) \quad
\mbox{and} \quad v(nr) = v(n) \gamma (r) $$ for all $r \in R, \ s
\in S, \ m \in M, \ n \in N$.

\vskip -0.3truecm \qquad Finally we show that (iii) in Proposition
\ref{propdefaut0} is satisfied. Let $m \in M, \ n \in N$. Applying
$\phi$ to
$$\left[ \begin{array}{cc}  0 & m \\
0 &0  \end{array}\right] \ \left[ \begin{array}{cc} 0 & 0 \\
n   &0 \end{array} \right] = \left[ \begin{array}{cc} [m,n] & 0 \\
0  & 0 \end{array}\right],$$ we obtain that
$$\left[ \begin{array}{cc} 0 & u (m) \\
0 &0 \end{array}\right] \ \left[ \begin{array}{cc} 0  &0\\
v(n) & 0  \end{array}\right] = \left[ \begin{array}{cc} \gamma ([m,n]) & \gamma ([m,n])m'_0 \\
n'_0 \gamma ([m,n]) & 0 \end{array}\right],$$ and so $[u(m), v(n)
] = \gamma ([m,n])$. Similarly,  $(v(n), u(m)) = \delta ((n,m))$.
We thus conclude that $\phi \in \mbox{Iso}_0^0(T,T')$.\\

\vskip -0.5truecm \qquad (2) It is clear that any $\phi \in
\mbox{Iso}_0^1(T,T')$ is an anti-semigraded isomorphism.
Conversely, let $\phi$ be an anti-semigraded isomorphism from $T$
to $T'$. If $\alpha_T :T\ra T^{-1}$ is the canonical isomorphism,
then $\phi \alpha_T^{-1}:T^{-1}\ra T'$ is a semigraded
isomorphism, so by the first part of the theorem we have that
$\phi \alpha_T^{-1} \in \mbox{Iso}_0^0(T^{-1},T')$. Then $\phi \in
\mbox{Iso}_0^0(T^{-1},T')\alpha_T=\mbox{Iso}_0^1(T,T')$ by the
first equality in equation (\ref{linkiso0iso1}).\end{proof}

\begin{remark} \label{remsemigraded}
Assume that $R,S,R'$ and $S'$ are indecomposable rings, and all of
$M,N,M'$ and $N'$ are zero. Then identifying $T=R\times S$ and
$T'=R'\times S'$, it is easy to see that any isomorphism
$\phi:T\ra T'$ is either of the form $\phi
(r,s)=(\gamma(r),\delta(s))$ for some isomorphisms $\gamma:R\ra
R'$ and $\delta:S\ra S'$, or of the form $\phi
(r,s)=(\sigma(s),\rho(r))$ for some isomorphisms $\rho:R\ra S'$
and $\sigma:S\ra R'$. Therefore we have that ${\rm Iso}(T,T')={\rm
Iso}_g(T,T')={\rm Iso}_g^+(T,T')={\rm Iso}_g^-(T,T')$, and this is
the set of all semigraded isomorphisms from $T$ to $T'$ (and also
the set of all anti-semigraded isomorphisms from $T$ to $T'$).
Also it is the disjoint union of ${\rm Iso}_0^0(T,T')$ and ${\rm
Iso}_0^1(T,T')$, so ${\rm Iso}(T,T')={\rm Iso}_g(T,T')={\rm
Iso}_0(T,T')$.
\end{remark}

\begin{corollary} \label{autgaut0}
If $R,S,R'$ and $S'$ are indecomposable, then ${\rm
Iso}_g(T,T')\subseteq {\rm Iso}_0(T,T')$.
\end{corollary}
\begin{proof} The result is clear if
$\mbox{Iso}_g(T,T')=\emptyset$. Assume now that
$\mbox{Iso}_g(T,T')$ is non-empty, implying that $T$ and $T'$ are
isomorphic. If not both $M'$ and $N'$ are zero, then the result
follows from Theorem \ref{semigraded}. If $M'=0$ and $N'=0$, then
we also must have $M=0$ and $N=0$ (otherwise $T$ is indecomposable
by Corollary \ref{corindecomposable}, while $T'$ is not so), and
the result follows from Remark \ref{remsemigraded}.
\end{proof}

\begin{corollary}
If $R,S,R'$ and $S'$ are indecomposable rings and the Morita
context with which $T$ (or $T'$) is asssociated is strict, then
${\rm Iso}_g(T,T')={\rm Iso}_0(T,T')$.
\end{corollary}
\begin{proof} It follows from Proposition \ref{autgstrict} and Corollary
\ref{autgaut0}.
\end{proof}

\begin{remark}
If $R$ and $S$ are not indecomposable, then ${\rm Aut}_g(T)$ is
not necessarily contained in ${\rm Aut}_0(T)$. Indeed, consider,
for example, the case where $R=S\times S$, $M=0$ and $N=0$. Then
we can identify $T$ with $S\times S\times S$ , and the map
$\phi:T\ra T$, $\phi (s_1,s_2,s_3)=(s_1,s_3,s_2)$, $s_1,s_2,s_3\in
S$, is an automorphism of $T$, which is clearly graded, since
$T=T_0$, but it is easy to see that $\phi\notin {\rm Aut}_0(T)$.
\end{remark}

\section{Automorphisms in the case of zero Morita maps}

\vskip -0.3truecm\qquad Let $T = \left[ \begin{array}{cc} R & M\\
 N & S \end{array} \right]$ and $T' = \left[ \begin{array}{cc} R' & M'\\
 N' & S' \end{array} \right]$ be Morita context rings. We have seen that in general
$\mbox{Iso}_0(T,T')$ is smaller than $\mbox{Iso}(T,T')$. Remark
\ref{remaut0aut} shows that if
$\mbox{Iso}_0(T,T')=\mbox{Iso}(T,T')\neq \emptyset$, then
necessarily all the Morita maps in both Morita contexts are zero.
In this section we show that the converse also holds, provided
that the rings $R$ and $S$ have only 0 and 1 as idempotents.

\vskip -0.5truecm \qquad Throughout this section $T$ and $T'$ will
be Morita context rings as above such that $R,S,R'$ and $S'$ have
only 0 and 1 as idempotents, and the Morita maps are zero in each
of the two contexts with which $T$ and $T'$ are associated.

\vskip -0.5truecm \qquad First note that a matrix $\left[ \begin{array}{cc} r & m \\
n &s \end{array} \right]$ in $T$ is idempotent if and only if $$
r^2 = r, \quad s^2 = s , \quad rm+ms = m \quad \mbox{and} \quad nr
+ sn = n.$$
 Therefore, apart from the trivial idempotents $\left[ \begin{array}{cc}0 & 0 \\
0 & 0 \end{array} \right]$ and $\left[ \begin{array}{cc}1 & 0 \\
0 & 1 \end{array} \right]$, the other idempotents of~$T$ are of
the form
\begin{equation} \label{8}
\left[ \begin{array}{cc}1 & m \\
n & 0 \end{array} \right] \quad \mbox{or} \quad \left[ \begin{array}{cc} 0 & m \\
n & 1 \end{array}\right]. \end{equation}

\vskip -0.3truecm The two types of idempotents in (\ref{8}) will
henceforth be called type 1 idempotents and type 2 idempotents
respectively.

\vskip -0.3cm \qquad We look at the action of an isomorphism
$\phi:T\ra T'$ on an idempotent $E\in T$ of the form
$$\left[ \begin{array}{cc} 1& m \\
0 & 0 \end{array} \right], \quad \left[ \begin{array}{cc} 1 & 0 \\
n & 0 \end{array} \right], \quad \left[ \begin{array}{cc} 0 & m \\
0 & 1 \end{array} \right] \quad \mbox{or} \quad \left[ \begin{array}{cc} 0& 0 \\
n & 1 \end{array} \right].$$ Note that the first two are type 1
idempotents, and the last two are type 2 idempotents. Since $\phi$
is injective and $\phi (E)$ is also an idempotent, we have that
$\phi (E)$ has one of the forms in (\ref{8}) inside the ring $T'$.

\begin{lemma} \label{lemma1}
If there exists $m_1\in M$ such that $\phi$ maps the type 1 idempotent $\left[ \begin{array}{cc} 1 & m_1\\
0 & 0 \end{array} \right]$ to a type~2 idempotent, then \vskip
-0.5cm
(a) $\phi$ maps every type 1 idempotent of the form $\left[ \begin{array}{cc} 1 & m \\
0 & 0 \end{array} \right]$ to a type 2 idempotent; and \vskip
-0.5cm (b) there is a fixed element $a'_2 \in M'$ and a map $v_2:
M \ra N'$ such that, for all $m \in M$,
$$\phi \left(\left[ \begin{array}{cc} 1& m\\
0 & 0 \end{array} \right] \right)= \left[ \begin{array}{cc} 0 & a'_2\\
v_2(m) & 1 \end{array}\right].$$
\end{lemma}

\begin{proof}
Let $\phi \left(\left[ \begin{array}{cc} 1 & m_1\\
0 & 0 \end{array}\right]\right) = \left[ \begin{array}{cc} 0 & m_1'\\
n_1' & 1 \end{array}\right]$ for some $m_1' \in M', \ n_1' \in
 N'$. \vskip -0.5cm
\qquad (a) Suppose there is an $m \in M$ such that $\phi \left(\left[ \begin{array}{cc} 1 & m\\
0 &0 \end{array}\right]\right)$ is a type 1 idempotent, i.e. $\phi \left(\left[ \begin{array}{cc} 1 & m\\
0 &0 \end{array} \right]\right) = \left[ \begin{array}{cc} 1 & m'\\
n' & 0 \end{array}\right]$ for some $m' \in M', \ n' \in N'$.
Since
$\left[ \begin{array}{cc} 1 & m_1\\
0 & 0 \end{array} \right] \ \left[ \begin{array}{cc} 1 & m \\
0 & 0 \end{array} \right] = \left[ \begin{array}{cc} 1 &m \\
0 &0 \end{array} \right]$, 
applying $\phi$ yields
\begin{equation}
\left[ \begin{array}{ll} 0 & m_1'\\
n_1' & 1 \end{array} \right] \ \left[ \begin{array}{ll} 1 & m' \\
n' & 0 \end{array} \right] = \left[ \begin{array}{ll} 1& m' \\
n' & 0 \end{array} \right]  \label{12}
\end{equation}
Equating the entries in position $(1,1)$ in (\ref{12}), we get
that $0=1$, a contradiction.  Hence, $\phi$ maps every type 1 idempotent of the form $\left[ \begin{array}{cc} 1 & m \\
0 & 0 \end{array} \right]$ to a type 2 idempotent.

(b)  By (a), there are maps $\alpha_2:  M \ra M'$ and $v_2: M \ra
N'$ such that
$$\phi \left(\left[ \begin{array}{cc} 1 & m \\
0 &0 \end{array} \right] \right)= \left[ \begin{array}{cc} 0 &
\alpha_2(m)\\ v_2 (m) & 1 \end{array}\right] $$ for all $m\in M$.
Let $m$ and $m_0$ be arbitrary elements of $M$. Applying $\phi$ to
$\left[ \begin{array}{cc} 1 & m_0\\
0 & 0 \end{array} \right] \ \left[ \begin{array}{cc} 1 & m \\
0 & 0 \end{array} \right] = \left[ \begin{array}{cc} 1 &m \\
0 &0 \end{array} \right]$, we obtain that
$$\left[ \begin{array}{cc} 0 & \alpha_2(m_0) \\
v_2 (m_0) & 1 \end{array} \right] \ \left[ \begin{array}{cc} 0 & \alpha_2 (m)\\
v_2 (m) &1 \end{array} \right] = \left[ \begin{array}{cc} 0 &\alpha_2(m)\\
v_2(m) &1 \end{array}\right].$$ Position $(1,2)$ shows that
$\alpha_2(m) = \alpha_2(m_0)$. Consequently, $\alpha_2$ is a
constant function, say $\alpha_2(m) = a_2'$ for some fixed $a_2'
\in M'$, which concludes the proof.
\end{proof}

\begin{corollary}\label{corollary2}
If there is an $m \in M$ such that $\phi$ maps the type 1 idempotent $\left[ \begin{array}{cc}1 & m \\
0 &0 \end{array}\right]$ to a type~1 idempotent, then \vskip
-0.5cm
(a) $\phi$ maps every type 1 idempotent of the form $\left[ \begin{array}{cc} 1 & m\\
0 &0 \end{array} \right]$ to a type 1 idempotent; and \vskip
-0.5cm (b) there is a fixed element $b'_1 \in N'$ and a map $u_1:
M \ra M'$ such that for all $m \in M$,
$$\phi \left(\left[ \begin{array}{cc} 1& m \\
0 &0 \end{array}\right]\right) = \left[ \begin{array}{cc} 1 & u_1(m) \\
b'_1 & 0 \end{array} \right].$$
\end{corollary}

\begin{proof} (a) This is an immediate consequence of Lemma \ref{lemma1}(a).

\vskip -0.3truecm \qquad (b) By (a), there are maps $u_1: M\ra M'$
and $\beta_1: M \ra N'$ such that
$$\phi \left(\left[ \begin{array}{cc} 1 & m\\
0 &0 \end{array} \right]\right) = \left[ \begin{array}{cc} 1 & u_1(m) \\
\beta_1(m) & 0 \end{array} \right]$$ for any $m\in M$. Proceeding
as in the proof of Lemma \ref{lemma1}(b), but using position
$(2,1)$ instead of position $(1,2)$, the desired result follows.
\end{proof}

\qquad The foregoing results are summarized in:
\begin{proposition} \label{proposition3}
Either there is a map $u_1: \ M \ra M'$ and a  fixed element $b'_1
\in N'$ such that
$\phi$ maps every type 1 idempotent of the form $\left[ \begin{array}{cc} 1 & m \\
0 &0 \end{array}\right]$ to a type 1 idempotent of the form $\left[ \begin{array}{cc} 1 & u_1(m) \\
b'_1 &0 \end{array} \right]$, or there is a map $v_2:  M \ra N'$
and a fixed element
$a'_2 \in M'$ such that $\phi$ maps every type 1 idempotent of the form $\left[ \begin{array}{cc} 1 & m \\
0 & 0 \end{array}\right]$ to a type 2 idempotent of the form $\left[ \begin{array}{cc} 0 & a'_2 \\
v_2 (m) & 1 \end{array}\right]$.
\end{proposition}

\qquad Similar arguments yield the following result for the action
of $\phi$ on type 1
idempotents of the form $\left[ \begin{array}{cc} 1 & 0 \\
n & 0 \end{array}\right].$

\begin{proposition} \label{proposition4}
For all $n \in N$, either $\phi \left(\left[ \begin{array}{cc} 1& 0 \\
n &0 \end{array}\right]\right) = \left[ \begin{array}{cc} 1 & a'_1 \\
v_1 (n) & 0 \end{array}\right]$ for some fixed $a'_1 \in M'$ and
some map $v_1: N\ra N'$,
or $\phi \left(\left[ \begin{array}{cc} 1 & 0 \\
n & 0 \end{array}\right]\right) = \left[ \begin{array}{cc} 0 & u_2(n) \\
b'_2 & 1 \end{array}\right]$ for some fixed $b'_2 \in N'$ and some
map $u_2 : N \ra M'.$
\end{proposition}

\qquad Setting $m = 0$ in $\left[ \begin{array}{cc} 1 & m \\
0 & 0 \end{array} \right]$ in Proposition \ref{proposition3} and setting $n =0$ in $\left[ \begin{array}{cc} 1 &0  \\
n &0 \end{array} \right]$ in Proposition~\ref{proposition4}, we
obtain in both cases
the type 1 idempotent $\left[ \begin{array}{cc} 1 &0 \\
0 & 0 \end{array} \right]$, and so if we consider
positions~$(1,1)$ and~$(2,2)$ of the
possible actions of $\phi$ on $\left[ \begin{array}{cc} 1 & 0 \\
0 & 0 \end{array} \right]$ as described in these two propositions,
we arrive at:

\begin{corollary} \label{corollary5}
For all $m \in M$ and all $n \in N$, either \vskip -0.3cm
$(I') \qquad \phi \left(\left[ \begin{array}{cc} 1 & m \\
0 &0 \end{array}\right]\right) = \left[ \begin{array}{cc} 1 & u_1(m) \\
v_1 (0) & 0 \end{array}\right] \quad \mbox{and} \quad \phi \left(\left[ \begin{array}{cc} 1 & 0 \\
n& 0 \end{array} \right]\right) = \left[ \begin{array}{cc} 1 & u_1(0) \\
v_1(n) & 0 \end{array}\right]$ \vskip -0.2cm \hspace*{1.2cm} for
some maps $u_1: M \ra M'$ and $v_1 : N \ra N'$, or \vskip -0.2cm
$(I'') \qquad \phi \left(\left[ \begin{array}{cc} 1 & m \\
0 &0 \end{array}\right]\right) = \left[ \begin{array}{cc} 0 & u_2(0) \\
v_2 (m) & 1 \end{array}\right] \quad \mbox{and} \quad \phi \left(\left[ \begin{array}{cc} 1& 0 \\
n & 0 \end{array}\right]\right)  = \left[ \begin{array}{cc} 0 & u_2(n) \\
v_2(0) & 1 \end{array}\right]$ \vskip -0.2cm \hspace*{1.2cm} for
some maps $u_2: N \ra M'$ and $v_2: M \ra N'$.
\end{corollary}

\qquad We set
$$
m'_0: = u_1(0) \quad \mbox{and} \quad n'_0 : = v_1
(0),\label{(13)}
$$
where $u_1$ and $v_1$ are as in Corollary \ref{corollary5}.

\vskip -0.3truecm\qquad Turning our attention to type 2 idempotents of the form $\left[ \begin{array}{cc} 0 & n \\
n & 1\end{array} \right]$ and $\left[ \begin{array}{cc} 0 &0 \\
n & 1 \end{array}\right]$, it can be verified as above that
\begin{corollary} \label{corollary6}
For all $m \in M$ and all $n \in N$, either \vskip -0.3cm
$(II') \qquad \phi \left(\left[ \begin{array}{ll} 0 & m \\
0 & 1 \end{array}\right]\right) = \left[ \begin{array}{cc} 0 & h_1(m) \\
p_1(0) & 1 \end{array}\right]$ \quad and \quad $\phi \left(\left[ \begin{array}{ll} 0 & 0 \\
n & 1 \end{array}\right]\right) = \left[ \begin{array}{cc} 0 & h_1(0)\\
p_1(n) & 1 \end{array}\right]$ \vskip -0.2cm \hspace*{1.5cm} for
some maps $h_1: M \ra M'$ and $p_1: N \ra N'$, or \vskip -0.2cm
$(II'') \qquad \phi \left(\left[ \begin{array}{ll}0 & m \\
0 &1 \end{array} \right]\right) = \left[ \begin{array}{cc} 1 & h_2(0) \\
p_2(m) & 0 \end{array}\right]$ \quad and \quad $\phi \left(\left[ \begin{array}{ll} 0 & 0 \\
n & 1 \end{array}\right]\right) = \left[ \begin{array}{cc} 1 & h_2(n) \\
p_2(0) & 0 \end{array}\right]$ \vskip -0.2cm \hspace*{1.5cm} for
some maps $h_2: N \ra M'$ and $p_2: M \ra N'$.
\end{corollary}

\begin{theorem} \label{lasttheorem}
Let $R,S,R'$ and $S'$ be rings having only trivial idempotents,
$M$ an $R$-$S$-bimodule, $N$ an $S$-$R$-bimodule, $M'$ an
$R'$-$S'$-bimodule and
$N'$ an $S'$-$R'$-bimodule. Let $T=\left[ \begin{array}{cc} R &M\\
 N& S \end{array} \right]$ and $T'=\left[ \begin{array}{cc} R' &M'\\
 N'& S' \end{array} \right]$ be the associated Morita context rings, where the contexts are
 considered with both Morita maps equal to zero.
 Then ${\rm Iso}_0(T,T')={\rm Iso}(T,T')$.
 \end{theorem}
 \begin{proof} The result is clear if
 $\mbox{Iso}(T,T')=\emptyset$. Assume that $\mbox{Iso}(T,T')$ is
 non-empty, i.e.~$T$ and $T'$ are isomorphic. If both $M$ and $N$
 are zero, then $T\simeq R\times S$ is not indecomposable. This
 forces $M'$ and $N'$ to be zero, otherwise $T'$ would be
 indecomposable by Corollary \ref{corindecomposable}. Now
 $\mbox{Iso}_0(T,T')=\mbox{Iso}(T,T')$ by Remark
 \ref{remsemigraded}.

\vskip-0.4truecm\qquad Assume now
 that at least one of $M$ and $N$ is non-zero. Then as above we can not have both
 $M'$ and $N'$ equal to zero.  Let $\phi:T\ra T'$ be a ring
 isomorphism. Since $\left[ \begin{array}{ll} 1 & 0 \\
0  & 0 \end{array} \right] + \left[ \begin{array}{ll} 0 & 0 \\
0 & 1\end{array} \right] = \left[ \begin{array}{ll} 1 & 0 \\
0 & 1\end{array} \right]$, it follows from Corollaries
\ref{corollary5}-\ref{corollary6}, by considering
position~$(2,2)$, that $(I')$ and $(II'')$ cannot simultaneously
hold; neither can $(I'')$ and $(II')$ simultaneously hold.
Therefore, either $(I')$ and $(II')$ are true, or $(I'')$ and
$(II'')$ are true.

\vskip -0.4truecm\qquad We first consider

\vskip -0.3truecm

\ul{$(I')$ and $(II')$}: For every $m \in M$,
\begin{eqnarray*}
\phi \left(\left[ \begin{array}{ll} 0 & m \\
0 & 0 \end{array}\right]\right) &= & \phi \left(\left[ \begin{array}{ll} 1 & m \\
0 & 0 \end{array}\right]\right) + \phi \left(\left[ \begin{array}{ll} 0 & 0 \\
0 &1 \end{array}\right]\right) - \phi \left(\left[ \begin{array}{ll} 1  & 0 \\
0 &1 \end{array}\right]\right)\\
\\
&= & \left[ \begin{array}{cc} 1 & u_1(m) \\
v_1(0) & 0 \end{array} \right] + \left[ \begin{array}{cc} 0 & h_1(0)\\
p_1(0) & 1 \end{array}\right] - \left[ \begin{array}{ll} 1   & 0 \\
0 & 1\end{array} \right]\\
\\
& = & \left[ \begin{array}{cc} 0 & u_1(m) + h_1(0)\\
v_1(0) + p_1(0) & 0 \end{array}\right],
\end{eqnarray*}
and setting $m=0$, we conclude that
$$h_1(0) = -m'_0.$$
Hence, defining $u: M \ra M'$ by $u(m) = u_1(m) - m'_0$, it
follows that
\begin{equation}
\phi \left(\left[ \begin{array}{cc} 0 & m \\
0 & 0 \end{array}\right]\right) = \left[ \begin{array}{cc} 0 & u(m) \\
0  & 0 \end{array}\right] \label{16}
\end{equation}
for all $m \in M$.

\vskip -0.3cm\qquad Similarly, considering the equality
$$\phi \left(\left[ \begin{array}{cc} 0 & 0 \\
n  & 0 \end{array} \right]\right) = \phi \left(\left[ \begin{array}{cc} 1 &0 \\
n & 0 \end{array}\right]\right) + \phi \left(\left[ \begin{array}{cc} 0 &0 \\
0 & 1 \end{array} \right]\right) - \phi \left(\left[ \begin{array}{cc} 1 & 0 \\
0 &1 \end{array}\right]\right)$$ and defining $v: N \ra N'$ by
$v(n) = v_1(n) - v_1(0)$, it follows as above that
\begin{equation}
\phi \left[ \begin{array}{cc} 0 & 0 \\
n  & 0 \end{array} \right] = \left[ \begin{array}{cc}0  & 0 \\
v(n) & 0 \end{array} \right], \label{18}
\end{equation}
for all all $n \in N$.

\vskip -0.5truecm \qquad Equations (\ref{16}) and (\ref{18}) show
that $\phi(T_1)\subseteq T_1$ and $\phi(T_{-1})\subseteq T_{-1}$.
Thus $\phi:T\ra T'$ is a semigraded isomorphism, hence $\phi \in
\mbox{Iso}_0^0(T,T')$ by Theorem \ref{semigraded}.

\vskip -0.3truecm \qquad The other possible situation is \vskip
-0.3truecm

\ul{$(I'')$ and $(II'')$}: Then $\psi=\phi \alpha_T^{-1}:T^{-1}\ra
T'$ is a ring isomorphism between the Morita context rings
$T^{-1}$ and $T'$, and it is easy to see that $\psi$ acts on the
particular idempotents (of $T^{-1}$) that we consider according to
the case $(I')$ and $(II')$. Then applying the result in the first
part of the proof, we find that $\psi
\in\mbox{Iso}_0^0(T^{-1},T')$. Now using (\ref{linkiso0iso1}) we
get that $\phi=\psi \alpha_T\in \mbox{Iso}_0^1(T,T')$.
\end{proof}

\begin{remark} \label{finalremark}
(1) We first note that any ring $A$ which has a non-trivial
idempotent $e$ is isomorphic to a Morita context ring. To see
this, we consider the Peirce decomposition associated with the
complete system of orthogonal idempotents $\{ e,1-e\}$; more
precisely,
$$A=eAe\oplus eA(1-e)\oplus (1-e)Ae\oplus (1-e)A(1-e)$$
as additive groups. Moreover, $eAe$ is a ring with identity
element $e$, $(1-e)A(1-e)$ is a ring with identity element $1-e$,
$eA(1-e)$ is a left $eAe$, right $(1-e)A(1-e)$-bimodule with
actions defined by the multiplication of the ring $A$, and
similarly $(1-e)Ae$ is a left $(1-e)A(1-e)$, right $eAe$-bimodule.
We have a Morita context $(eAe, (1-e)A(1-e), eA(1-e), (1-e)Ae,
f,g)$, where $f$ and $g$ are induced by the multiplication of $A$.
It is easy to see that the above decomposition as additive groups
induces in fact an isomorphism of rings
$$A\simeq \left[ \begin{array}{cc} eAe & eA(1-e)\\
 (1-e)Ae & (1-e)a(1-e) \end{array} \right].$$

\vskip -0.3truecm \qquad A ring $T=\left[ \begin{array}{cc} R & M\\
 N & S \end{array} \right]$ associated with a general Morita
 context is an instance of a ring with non-trivial idempotents,
 since $e:=\left[ \begin{array}{cc} 1 & 0 \\
0  & 0 \end{array} \right]$ is such an idempotent. It is clear
that $$eTe=\left[ \begin{array}{cc} R & 0 \\
0  & 0 \end{array} \right]\simeq R,\; (1-e)T(1-e)=\left[ \begin{array}{cc} 0 & 0 \\
0  & S \end{array} \right]\simeq S,$$
$$eT(1-e)=\left[ \begin{array}{cc} 0 & M \\
0  & 0 \end{array} \right]\simeq M,\; (1-e)Te=\left[ \begin{array}{cc} 0 & 0 \\
N  & 0 \end{array} \right]\simeq N.$$

\vskip -0.3truecm \qquad We conclude that a ring has non-trivial
idempotents if and only if
it is isomorphic to a Morita context ring.\\

(2) In [AW] a ring $R$ is called strongly indecomposable if it is
not isomorphic to a ring of the form $\left[ \begin{array}{cc} A & X\\
0 & B \end{array} \right]$, where $A$ and $B$ are rings, and $X$
is a nonzero left $A$, right $B$-bimodule.
The automorphism group of a generalized triangular matrix ring $\left[ \begin{array}{cc} R & _RM_S\\
0 & S \end{array} \right]$ over strongly indecomposable rings $R$
and $S$ is computed in [AW].

\vskip -0.3truecm \qquad As we explained in the first part of the
remark, a ring does not have non-trivial idempotents if and only
if it is not isomorphic to a Morita context ring, so this is a
4-corner version of the "strongly indecomposable ring" concept of
[AW]. Thus our Theorem \ref{lasttheorem} can be seen as a result
similar to Theorem 3.2 of [AW] for 4-corners generalized matrix
rings.
\end{remark}

\vskip -0.5cm

\section*{References}
\begin{enumerate}
\item [{[AW]}] P.N. \'Anh and L. van Wyk, {\it Automorphism groups
of generalized triangular matrix rings}, manuscript.
\item[{[BHKP]}] G.F. Birkenmeier, H.E. Heatherly, J.Y. Kim and
J.K. Park, {\it Triangular matrix representations},
J.~Algebra~{\bf 230} (2000), 558-595. \item[{[BK]}] G.P. Barker
and T.P. Kezlan, {\it Automorphisms of algebras of upper
triangular matrix rings}, Arch.~Math.~{\bf 55} (1990), 38-43.
\item[{[C1]}] S.P. Coelho, {\it The automorphism group of a
structural matrix algebra}, Linear Algebra Appl.~{\bf 195} (1993),
35-58. \item[{[C2]}] S.P. Coelho, {\it Automorphism groups of
certain algebras of triangular matrices}, Arch.~Math. (Basel) {\bf
61} (1993), 119-123. \item[{[D]}] W. Dicks, {\it Automorphisms of
the free algebra of rank two}, Group Actions on Rings (Brunswick,
Maine, 1984), 63-68, Contemp. Math., {\bf 43}, Amer. Math. Soc.,
Providence, RI, 1985. \item[{[DW]}]  S. D\u asc\u alescu and L.
van Wyk, {\it Do isomorphic structural matrix rings have
isomorphic graphs?}, Proc.~Amer.~Math.~Soc.~{\bf 124} (1996),
1385--1391. \item[{[F]}] C. Faith, {\it Algebra II: Ring Theory},
Springer, 1976. \item[{[I]}] I.M. Isaacs, {\it Automorphisms of
matrix algebras over commutative rings}, Linear Algebra Appl.~{\bf
31} (1980), 215-231. \item[{[J1]}] S. J\o ndrup, {\it
Automorphisms of upper triangular matrix rings},
Arch.~Math.~(Basel) {\bf 49} (1987), 497-502. \item[{[J2]}] S. J\o
ndrup, {\it The group of automorphisms of certain subalgebras of
matrix algebras}, J.~Algebra~{\bf 141} (1991), 106-114.
\item[{[Ke]}] T.P. Kezlan, {\it A note on algebra automorphisms of
triangular matrices over commutative rings}, Linear Algebra
Appl.~{\bf 135} (1990), 181-184. \item[{[Ko]}] M. Koppinen, {\it
Three automorphism theorems for triangular matrix algebras},
Linear Algebra Appl.~{\bf 245} (1996), 295-304. \item[{[KDW]}] R.
Khazal, S. D\u asc\u alescu and L. van Wyk, {\it Isomorphism of
generalized triangular matrix-rings and recovery of tiles},
Internat.~J.~Math.~Math.~Sci.~{\bf 2003}(9) (2003), 533-538.
\item[{[MCR]}] J.C. McConnell and J.C. Robson, {\it Noncommutative
Noetherian Rings}, John Wiley \& Sons, 1987. \item[{[NVO]}] C. N\u
ast\u asescu, F. van Oystaeyen, {\it Methods of graded rings},
Lecture Notes in Math., vol. 1836 (2004), Springer Verlag.
\item[{[R]}] J.J. Rotman, {\it Advanced Modern Algebra}, Prentice
Hall, New Jersey, 2002. \item[{[RZ]}] A. Rosenberg and D.
Zelinsky, {\it Automorphisms of separable algebras}, Pacific
J.~Math.~{\bf 11} (1961), 1109-1117.

\end{enumerate}

\end{document}